\documentclass[11pt]{amsart}
\usepackage{latexsym}
\usepackage{amsfonts}
\usepackage{xcolor}
\usepackage{graphicx}
\usepackage{amsmath,amssymb}

\newcommand{\field}[1]{\mathbb{#1}}
\newcommand{\C}{\field{C}}

\newtheorem{defi}{Definition}[section]

\newtheorem{theo}[defi]{Theorem}
\newtheorem{co}[defi]{Corollary}

\newtheorem{re}[defi]{Remark}
\newtheorem{ex}[defi]{Example}

\font\tenmsy=msbm10

\def\Bbb#1{\hbox{\tenmsy#1}} 
\title[Bi-Lipschitz characterization of space curves]{Bi-Lipschitz characterization of space curves}
\makeatletter

\@addtoreset{equation}{section}
\makeatother

\author{Alexandre Fernandes}
\address[Alexandre Fernandes]{Departamento de Matem\'atica\\ Universidade Federal do Cear\'a, Av.
Humberto Monte, s/n Campus do Pici  60440-900\\ Brasil.}
\email{alex@mat.ufc.br}

\author{Zbigniew Jelonek}
\address[Zbigniew  Jelonek]{Instytut Matematyczny\\
Polska Akademia Nauk\\
\'Sniadeckich 8, 00-656 Warszawa\\
Poland }
\email{najelone@cyf-kr.edu.pl}


\makeatletter
\@namedef{subjclassname@2020}{%
	\textup{2020} Mathematics Subject Classification}
\makeatother

\thanks{Alexandre Fernandes is  partially supported by CNPq-Brazil grant 304700/2021-5. Zbigniew Jelonek is partially supported by the grant of Narodowe Centrum Nauki number 2019/33/B/ST1/00755. }

\date{}

\begin{document}

\maketitle

\bibliographystyle{alpha}

\begin{abstract}
In the paper \cite{renato} Renato Targino shows that bi-Lipschitz type of plane curve is determined by the local ambient topological properties of curves. Here we show that it is not longer true in higher dimensions. However we show that bi-Lipschitz type of space curves is determined by the number of singular points and by the local ambient topological type of a generic projection of such curves into the affine plane. 
\end{abstract}

\section{Introduction}
Let $\Lambda, \Gamma\subset \C^n$ be two algebraic curves. In general if the germs $(\Lambda, p)$ and $(\Gamma, q)$ are {\it ambient topologically equivalent}, in the sense that there exists a germ of homeomorphism $\varphi\colon (\C^n,p)\rightarrow (\C^n,q)$ such that $\varphi (\Lambda,p)=(\Gamma,q)$, then this does not imply that the germs $(\Lambda, p)$ and $(\Gamma, q)$ are {\it bi-Lipschitz equivalent} in the sense that there exists a germ of bi-Lipschitz homeomorphism $\varphi\colon (\Lambda,p)\rightarrow (\Gamma,q)$ (with respect to the outer metric). Moreover we have even the following global result:

\begin{ex}
Let $n>2$. For every  irreducible singular algebraic curve  $\Gamma\subset \C^n$,  there  is an algebraic curve $\Lambda$, such that curves  $\Gamma, \Lambda$  are not bi-Lipschitz  equivalent but they are topologically equivalent.

\vspace{5mm}
{\rm Indeed, let $\Lambda$ be the normalization of $\Gamma$. Hence $\Lambda$ is an algebraic curve and  still can be embedded into $\C^n$ as smooth algebraic curve. But $\Lambda$ has no singularities  and by \cite{Bir}, \cite{Sam}  curves  $\Gamma, \Lambda$  are not bi-Lipschitz  equivalent. We have a canonical semi-algebraic  mapping $\Phi: \Lambda\to\Gamma,$ which by the assumption is a homeomorphism. Now, by Theorem 6.6 in \cite{jel} we can extend the mapping 
$\phi: \Lambda\to \Gamma$ to a global semi-algebraic homeomorphism $\Phi: \C^n\to\C^n.$}
\end{ex}

The situation is different in the dimension $n=2.$ In this dimension two germs are ambient topologically equivalent if and only if they are 
outer bi-Lipschitz equivalent (See \cite{p-t}, \cite{fer}, \cite{n-p}). In fact, in the paper \cite{renato},  Renato Targino classifies plane algebraic curves, up to global bi-Lipschitz homeomorphisms, in terms of local ambient topological properties of the projective closure of those curves. More precisely, two algebraic curves $\Lambda, \Gamma\subset \C^n$ are said {\it bi-Lipschitz equivalent} if there exists a bi-Lipschitz homeomorphism $\varphi\colon\Lambda\rightarrow\Gamma$ (with respect to the outer Euclidean metric induced from $\C^n$). In the paper \cite{renato},  Renato Targino shows the Theorem \ref{renato} below. 

\bigskip

\noindent{{\bf Notation.} Consider $\C^2$ embedded into $\Bbb {CP}^2$. Given an algebraic plane curve $X\subset\C^2$, let $\overline{X}\subset \Bbb{C P}^2$ be the respective projective closure of $X$ and let $\pi_{\infty}$ denote the line at infinity in $\Bbb{C P}^2$.}

\begin{theo}[\cite{renato}, Theorem 1.6]\label{renato}
	
Let $\Lambda, \Gamma\subset \C^2$ be two algebraic plane curves with irreducible components  $\Lambda = \bigcup_{i\in I} \Lambda_j$ and $\Gamma = \bigcup_{j\in J} \Gamma_i$.   The following statements are mutually equivalent.
\begin{enumerate}
	\item The curves $\Lambda$ and $\Gamma$ are bi-Lipschitz equivalent.
	\item There are bijections $\sigma \colon I\rightarrow J$ and $\rho$ between the set of singular points of $\overline{\Lambda}\cup\pi_{\infty}$ and the set of singular points of $\overline{\Gamma}\cup\pi_{\infty}$ such that $\rho(p)\in\pi_{\infty}$ if and only if $p\in\pi_{\infty}$, $(\overline{\Lambda}\cup\pi_{\infty}, p)$ is topologically equivalent to $(\overline{\Gamma}\cup\pi_{\infty},\rho(p))$, $(\overline{\Lambda}_i\cup\pi_{\infty}, p)$ is topologically equivalent to $(\overline{\Gamma}_{\sigma(i)}\cup\pi_{\infty},\rho(p))$ ($\forall i\in I$), for all singular point $p$ of  $\overline{\Lambda}\cup\pi_{\infty}$.
\end{enumerate}

\end{theo}

In this note we address the global classification problem given by bi-Lipschitz equivalence of algebraic space curves, i. e., algebraic curves in $\C^n$ ($n>2$). As in the local case, we obtain a characterization of the bi-Lipschitz equivalence classes of algebraic space curves by looking for their generic plane projections. In this direction, we point out that: given an algebraic curve $\Lambda$ in $\C^n$, any two generic plane projection are bi-Lipschitz equivalent, and this is why we can refer to a generic plane projection of $\Lambda$ as {\it the generic plane projection of} $\Lambda$. We are ready to state the main results of the paper.

\begin{theo}\label{main1} Two irreducible algebraic curves in $\C^n$ are bi-Lipschitz equivalent if and only if they have the same number of singular points and their generic plane projections are bi-Lipschitz equivalent.
\end{theo}

As a direct consequence of Theorem \ref{main1} and Theorem \ref{renato} (Theorem 1.6 of \cite{renato}) we get a characterization of bi-Lipschitz equivalence classes of irreducible algebraic space curves in terms of the local ambient topology of its generic plane projection.

\begin{co}\label{co}Let $\Lambda, \Gamma\subset \C^n $ ($n>2$) be two irreducble algebraic  curves. Let $X_{\Lambda}$ and $X_{\Gamma}$ denote their generic plane projections (respectively).  The following statements are mutually equivalent:
	\begin{enumerate}
		\item The curves $\Lambda$ and $\Gamma$ are bi-Lipschitz equivalent.
		\item  The curves $\Lambda$ and $\Gamma$ have the same number of singular points and there is a bijection $\rho$ between the set of singular points of $\overline{X}_{\Lambda}\cup\pi_{\infty}$ and the set of singular points of $\overline{X}_{\Gamma}\cup\pi_{\infty}$ such that $\rho(p)\in\pi_{\infty}$ if and only if $p\in\pi_{\infty}$ and $(\overline{X}_{\Lambda}\cup\pi_{\infty}, p)$ is topologically equivalent to $(\overline{X}_{\Gamma}\cup\pi_{\infty},\rho(p))$,

	\end{enumerate}	
\end{co}

\section{Main Result}
Let $B^k(R)\subset \C^k$ denote the $2k$ real dimensional  Euclidean ball of radius $R$ and center at $0.$

\begin{theo}\label{infinity}
Let $n>2$. If $X\subset \C^n$ is a closed algebraic curve,  then there are a real number $r>0$ and the proper projection 
$\pi: X\to\C^2$ such that  $\pi: X\setminus \pi^{-1}(B^2(r))\to \C^2\setminus B^2(r)$ is a bi-Lipchitz embedding.
\end{theo}

\begin{proof}
Of course it is enough to prove our theorem for a projection $\pi: X\to\C^{n-1}$ and then use induction on the number $n$. Consider $\C^n$ embedded into $\Bbb {CP}^n$. Let $X'$ be the projective closure of $X$ in $\Bbb {CP}^n.$

Let $Z:=X'\setminus X=\{ z_1,...,z_r\}.$ For $i\not=j$, let us denote by $L_{ij}$ the line $\overline{z_i,z_j}$ and let us denote by $\pi_\infty$ the hyperplane at infinity of $\C^n.$ Thus $\pi_\infty\cong \Bbb {C P}^{n-1}$ is a projective space of dimension $n-1.$ For a non-zero vector $v\in \C^n$, let $[v]$ denote the corresponding point in $\pi_\infty.$

Let $\Delta =\{ (x,y)\in  {X}\times {X} : x=y\}.$ Consider the mapping $$A:
{X}\times {X}\setminus \Delta \ni (x,y)\to
[x-y]\in \pi_\infty.$$ 

Let $\Gamma$ be the graph of $A$ in $X'\times X'\times {\pi_\infty}$ and take $\Gamma':=\overline{\Gamma}$
(we take this closure in $X'\times X' \times {\pi_\infty}$).  Let 
$p:\Gamma' \to \pi_\infty$ and
$q:\Gamma' \to X'\times X'$ be the canonical projections. Note that, for $z_i\in Z$, the set $q^{-1}(z_i,z_i)=Z_{i}$ is an algebraic set of dimension at most $1$.   Let  $Z_i'=p(Z_i).$ The set $W:=\bigcup Z_i'$ is a closed subset of $\pi_\infty$ of dimension  at most $1$, hence   $\pi_\infty\setminus W \not=\emptyset.$ Let $Q\in \pi_\infty\setminus (W \cup \bigcup L_{ij})$, since $W$ is closed, there exists a small ball $B_1\subset \pi_\infty$ with center at $Q$ such that $B_1\cap (W\cup \bigcup L_{ij})=\emptyset.$

Let $B^n(R)$ be a large ball in $\C^n$ and take  $V(R)=(X'\setminus B^n(R))\times (X'\setminus B(R))$ and let $O_R=q^{-1}(V(R)).$ Hence  $O_R$ is a neighborhood of $W$ in
$\pi_\infty$. We show that for $R$ sufficiently large, if $x,y\in X\setminus B^n(R),\ x\not=y$, then $A(x,y)\not\in B_1.$
Indeed, in other case, take $R_k=k \to \infty.$ Hence for every $k\in \Bbb N$ we have points $x_k,y_k\in X\setminus B^n(R),$ such that $A(x_k,y_k)\in B_1.$
But then $x_k,y_k\to \infty$, this means that we can assume that $x_k\to z_i$ and $y_k\to z_j.$ If $z_i=z_j$, then  $\lim A(x_k,y_k)=\lim p((x_k,y_k, [x_k-y_k]))\in p(Z_i)=Z_i'.$ It is a contradiction. If $z_i\not=z_j$, then $\overline{x_k,y_k}\to\overline{z_i,z_j}=L_{ij}$ and this means that $L_{ij}\cap B_1\not=\emptyset$, a contradiction again. 

Hence, there is a number $R$ sufficiently large, such that  if  $x,y\in X\setminus B(R)), \ x\not=y$, then $A(x,y)\not\in B_1.$ Let $\Sigma=A((X\setminus B(R))\times (X\setminus B(R))\setminus \Delta).$ Then $Q\not\in \overline{\Sigma}.$ Take a hyperplane $H\subset \C^n$ in this way, such that $Q\not\in \overline{H}.$ Of course $H\cong \C^{n-1}.$ Let $\pi:\C^n\to H$ be the projection with center $Q$ and let $K=\pi(X\cap B^n(R)).$ It is a compact set, hence there exists a ball $B^{n-1}(r)$ such that $K\subset B^{n-1}(r).$ 

Consider the proper mapping $\pi:  X\setminus \pi^{-1}(B^{n-1}(r))\to H\setminus B^{n-1}(r)$. We show that the projection $\pi$  is a bi-Lipschitz embedding. Indeed, since a complex linear isomorphism is a bi-Lipschitz mapping, we can assume that 
$Q=(0:0:...0:1)$ and $H=\{x_n=0\}.$ Of course $||p(x)-p(y)||\le ||x-y||.$ Assume that $p$ is not bi-Lipschitz, i. e., there is a sequence of points $x_j,y_j\in X\setminus \pi^{-1}(B^{n-1}(r))$ such that $$\frac{||p(x_j)-p(y_j)||}{||x_j-y_j||}\to 0,$$ as $n\to \infty.$ Let $x_j-y_j=(a_1(j),...,a_{n-1}(j),b(j))$ and denote by $P_j$ the corresponding point $(a_1(j):...:a_{n-1}(j):b(j))$ in $\Bbb {CP}^{n-1}.$ Hence $$P_j=\frac{(a_1(j):...:a_{n-1}(j):b(j))}{||x_j-y_j||}.$$ Since $\displaystyle\frac{(a_1(j),...,a_{n-1}(j))}{||x_j-y_j||}=
\frac{p(x_j)-p(y_j)}{||x_j-y_j||}\to 0$, we have that $P_j\to Q.$ It is a contradiction.

\end{proof}

 In the sequel we will use the following theorem of Jean-Pierre Serre (see \cite{mil}, p. 85):

\begin{theo}\label{serre}
If $\Gamma$ is an irreducible curve of degree $d$ and genus $g$  in the complex projective plane,
then $$\frac{1}{2} (d-1)(d-2)= g + \sum_{z\in Sing(\Gamma)} {\delta}_z,$$
where $\delta_z$ denotes the delta invariant of a point $z$.
\end{theo}

Before starting to prove Theorem \ref{main1}, let us introduce the notion of Euclidean subsets being bi-Lipschitz equivalent at infinity.

\begin{defi}
	Two subsets $X\in\C^n$ and $Y\in\C^m$ are called {\bf bi-Lipschitz equivalent at infinity} if there exist compact subsets $K_1\in\C^n$ and $K_2\in\C^m$ and a bi-Lipschitz homeomorphism $X\setminus K_1\rightarrow Y\setminus K_2.$ 
\end{defi}

\begin{re}\label{remark}
{\rm In order to prove that two algebraic plane curves $X$ and $Y$ are bi-Lischitz equivalent, Renato Targino (see \cite{renato}) showed that is enough to verify the following two conditions:
\begin{enumerate}
	\item There is a bijetion $\varphi\colon Sing(X)\rightarrow Sing(Y)$ such that $(X,p)$ is bi-Lipschitz equivalent to $(Y,\varphi(p))$ as germs, $\forall p\in Sing(X)$;
	\item $X$ and $Y$ are bi-Lipschitz equivalent at infinity.
\end{enumerate} 
Note that the proof of Renato Targino still works  in the case where $X$ and $Y$ are algebraic curves in $\C^n$ (not necessarily $n=2$).}
\end{re}

Now we can prove  Theorem \ref{main1}.


\begin{proof}[Proof of Theorem \ref{main1}] 
Let us suppose that $\Lambda$ and  $\Gamma$ are bi-Lipschitz equivalent, in particular, they have the same genus and they are bi-Lipschitz equivalent at infinity. We are going to prove that their generic plane projections $\Lambda'$ and $\Gamma'$, respectively, satisfy conditions 1) and 2) of Remark \ref{remark}. By using Theorem \ref{infinity}, we see that $\Lambda$ and $\Gamma$ are bi-Lipschitz equivalent at infinity to $\Lambda'$ and $\Gamma'$, respectively. Since, $\Lambda$ and  $\Gamma$ are bi-Lipschitz equivalent, it follows that $\Lambda'$ and $\Gamma'$ are bi-Lipschitz equivalent at infinity as well ($\Lambda'$ and $\Gamma'$ satisfy  item 2) of Remark \ref{remark}).

Before starting to show that $\Lambda'$ and $\Gamma'$ satisfy item 1) of Remark \ref{remark}, let us do some remarks about singularities of plane generic projections $X'$ of a space curve $X$ in $\C^n$. We have a partition of the singular subset $Sing(X')$ into two types of singularities: singularities that come from singularities of $X$ via the associated linear generic projection $X\subset\C^n\rightarrow X'\subset\C^2$ (let us denote the set of such singularities by $S_1(X')$) and the so-called {\bf new nodes} which are singularities that come from double-points of the associated linear generic projection $X\subset\C^n\rightarrow X'\subset\C^2$ (let us denote the set of new nodes by $S_2(X')$).

We resume our proof that  $\Lambda'$ and $\Gamma'$ satisfy item 1) of Remark \ref{remark}. It is clear that the local composition of the bi-Lipschitz homeomorphism $\Lambda'\rightarrow\Gamma'$ and the linear generic projections $\Lambda\subset\C^n\rightarrow \Lambda'\subset\C^2$ and $\Gamma\subset\C^n\rightarrow \Gamma'\subset\C^2$ gives a natural bijection $\varphi\colon S_1(\Lambda')\rightarrow S_1(\Gamma')$ such that, $(\Lambda',p)$ is bi-Lipschitz equivalent to $(\Gamma',\varphi(p))$ as germs, $\forall p\in S_1(\Lambda')$. Next, we are going to extend $\varphi$ to the set of new nodes. Notice that, since $\Lambda$ and $\Gamma$ have the same number of singular points and the same genus, we can deduce by Theorem  \ref{serre} that the number of new nodes which appear in $\Lambda'$ and $\Gamma'$ is the same in both cases. Indeed, since $\Lambda'$ and $\Gamma'$ are bi-Lipschitz equivalent at infinity they have the same degree (see Corollary 3.2 in \cite{bfs}) and they have topologically equivalent germs at infinity (as stated in Theorem 1.5 of \cite{renato}). Moreover $\Gamma'$ and $\Lambda'$ also have the same genus. Now, by Serre's Formula (Theorem \ref{serre}) we see that the number of new nodes must be the same in both cases. Since any two nodes are bi-Lipschitz equivalent as germs, we can consider $\varphi\colon S_2(\Lambda')\rightarrow S_2(\Gamma')$ as being any bijection such that $(\Lambda',p)$ is bi-Lipschitz equivalent to $(\Gamma',\varphi(p))$ as germs, $\forall p\in S_2(\Lambda')$. In other words, according to Remark \ref{remark}, we have proved that $\Lambda'$ and $\Gamma'$ are bi-Lipschitz equivalent. 

 On the other hand, let us suppose that the generic plane projections $\Lambda'$ (of $\Lambda $) and $\Gamma'$ (of $\Gamma$) are bi-Lipschitz equivalent. Thus, by using Theorem \ref{infinity}, we see that $\Lambda$ and $\Gamma$ are bi-Lipschitz equivalent at infinity, i.e., they satisfy item 2) of Remark \ref{remark}. Concerning to item 1) of Remark \ref{remark}, we have natural bijections $\varphi_{\Lambda}\colon Sing(\Lambda)\rightarrow S_1(\Lambda')$ and $\varphi_{\Gamma}\colon Sing(\Gamma)\rightarrow S_1(\Gamma')$ such that $(\Lambda,p)$ (respectively $(\Gamma,q)$) is bi-Lipschitz equivalent to $(\Lambda',\varphi_{\Lambda}(p))$ (respectively $(\Gamma',\varphi_{\Gamma}(q))$) as germs. Now, via the bi-Lipschitz homeomorphism between $\Lambda'$ and $\Gamma'$, the linear generic projections $\Lambda\subset\C^n\rightarrow \Lambda'\subset\C^2$ and $\Gamma\subset\C^n\rightarrow \Gamma'\subset\C^2$ are local bi-Lipschitz homeomorphisms, we have a natural bijection $\varphi'\colon S_1(\Lambda')\rightarrow S_1(\Gamma')$ such that $(\Lambda',p)$  is bi-Lipschitz equivalent to $(\Gamma',\varphi'(p))$ as germs, $\forall p\in \Lambda'$. Finally, by the composite mapping $\varphi_{\Gamma}^{-1}\circ\varphi'\circ\varphi_{\Lambda}$, we conclude that $\Lambda$ and $\Gamma$ satisfy item 1) of Remark \ref{remark}.
\end{proof}

    \end{document}